\documentclass[12pt]{amsart}
\usepackage[margin=1in]{geometry}
\usepackage{lmodern}
\usepackage{mathptmx}
\usepackage{mathtools}
\usepackage{amsthm}
\usepackage{amssymb}
\usepackage{amsrefs}
\usepackage{graphicx}
\usepackage{tikz}
	\usetikzlibrary{decorations.pathreplacing}
	\usetikzlibrary{patterns}
\usepackage{caption}
\usepackage{csquotes}
\usepackage{chngcntr}

\counterwithin{figure}{section}
\MakeOuterQuote{"}
\newcommand{\ra}{\rightarrow}
\newcommand{\pr}{\prime}
\newcommand{\de}{\partial}
\newcommand{\te}{\theta}

\newcommand{\R}{\mathbb{R}}
\newcommand{\Z}{\mathbb{Z}}

\newcommand{\abs}[1]{\left\lvert #1 \right\rvert}
\newcommand{\tld}[1]{\widetilde{#1}}
\newcommand{\lbar}[1]{\overline{#1}}
\DeclareMathOperator{\id}{id}

\renewcommand{\coprod}{\rotatebox[origin = c]{180}{$\prod$}}

\newcommand*{\eqqcolon}{\rotatebox[origin=c]{-180}{$\coloneqq$}}

\newtheorem{thm}{Theorem}

\theoremstyle{definition}

\newtheorem*{defin*}{Definition}
\newtheorem{claim}{Claim}[section]
\theoremstyle{plain}
\newtheorem{lemma}[thm]{Lemma}

\numberwithin{equation}{section}

\begin{document}

\title{The 2-width of Embedded 3-manifolds}

\author{Michael Freedman}
\address{\hskip-\parindent
	Michael Freedman \\
    Microsoft Research, Station Q, and Department of Mathematics \\
    University of California, Santa Barbara \\
    Santa Barbara, CA 93106 \\}

\begin{abstract}
	We discuss a possible definition for "$k$-width" of both a closed $d$-manifold $M^d$, and on embedding $M^d \overset{e}{\hookrightarrow} \R^n$, $n > d \geq k$, generalizing the classical notion of width of a knot. We show that for every 3-manifold 2-width$(M^3) \leq 2$ but that there are embeddings $e_i: T^3 \hookrightarrow \R^4$ with 2-width$(e_i) \ra \infty$. We explain how the divergence of 2-width of embeddings offer a tool to which might prove the Goeritz groups $G_g$ infinitely generated for $g \geq 4$. Finally we construct a homeomorphism $\te_g: G_g \ra \mathrm{MCG}(\underset{g}{\#} S^2 \times S^2)$, suggesting a potential application of 2-width to 4D mapping class groups.
\end{abstract}

\maketitle

\section{Introduction}

For smooth knots $K: S^1 \hookrightarrow S^3$ the simplest definition of width$(K)$ is $\mathrm{minimax}\abs{K \cup \pi^{-1}(x)}$, where the minimum is taken over all embeddings (still denoted $K$) isotopic to $K$ and the maximum is over all $x \in \R$. $\pi$ is projection (to, say, the third coordinate) $\pi: \R^3 \ra \R$, $\pi(x_1, x_3, x_3) = x_3$. We assume $K$ lies in $\R^3 = S^3 \setminus$pt, and that $x$ is a regular value of the composition $S^1 \xrightarrow{K} \R^3 \xrightarrow{\pi} \R$. Finally $\abs{\ \cdot\ }$ counts number of points.

In \cite{gabai} a slightly more refined version of width, called "Gabai width," was introduced in the proof that all knots satisfy property R. A similar concept of width by Thompson was used in \cite{thompson} to produce an algorithm for recognizing the three sphere, $S^3$. In \cite{st} the width of three-manifolds was introduced by Scharlemann and Thompson.

In [F1] the "width of a group" was studied also using maps to $\R^1$ and in \cite{fh} the "width" of knots $S^n \overset{K}{\hookrightarrow} \R^{n+2}$ was investigated via a projection $\R^{n+2} \xrightarrow{\pi} \R^n$. Let's try to make a couple general definitions consistent with these early uses of width.

We'll use the term $k$-width for minimaxes involving projection to $\R^k$. So the classical notions are kinds of "1-width," whereas the width discussed in \cite{fh} would now be called $n$-width. Our primary focus is on the width of embeddings, specifically the 2-width of $e: M^3 \hookrightarrow \R^4$ embedding of a 3-manifold in $\R^4$. But it is natural to discuss the notion of width first in the absolute context first---no embedding, just a closed manifold $M$. This will give some perspective on what facts about $k$-width of embedding are interesting or even surprising.

Let $M^d$ be a smooth closed manifold of dimension $d$. All constructions and terminology will be in the smooth category in this paper and manifolds are assumed closed unless otherwise stated so these conditions will not be repeated. Given a numerical (real number) property $P$ of $d-k$ manifolds we can define $k$-$\text{width}_P(M^d)$
\begin{equation}
	k\text{-width}_P(M)^d = \mathrm{minimax}\ P(\pi^{-1}(x))
\end{equation}
where the minimum is taken over an appropriate class of maps $\pi: M^d \ra \R^k$ and the maximum is taken over all regular values $x \in \R^k$.

I thank D. Zuddas for pointing out that some genericity assumption on $\pi: M^d \ra \R^k$ is required. Otherwise, for example a constant map would have only the empty set among preimages of its regular values. A natural choice, and let us adopt it, is topologically stable (t.s.) maps\footnote{An alternative, less restrictive, assumption on the smooth map $\pi: M^d \ra \R^k$ , is that all their preimages be smoothly stratified spaces of dimension less than or equal to $\max(0,d-k)$. This restriction has the advantage of being an easily verified hypothesis in examples. The exact definition of the “appropriate” class of maps has no effect on either of our two theorems.}. It is an unpublished theorem of John Mather that, without dimension restriction, such maps $\{f: P \ra Q\}_{\text{t.s.}}$, $P$ compact, are open dense in the space of all smooth maps $\{f: P \ra Q\}$ with the Whitney $C^\infty$-topology.

\begin{defin*}
	$f: P \ra Q$ is \emph{topologically stable} if for any $\tld{f}$ sufficiently close to $f$ in the $C^\infty$-topology there exist homeomorphisms $h_1$ and $h_2$ making the following diagram commute:

	\begin{center}
		\begin{tikzpicture}[scale=1.3]
			\node at (0,0) {$P$};
			\node at (1.25,0.1) {$\scriptstyle{f}$};
			\draw [->] (0.3,0) -- (2.2,0);
			\node at (2.5,0) {$Q$};

			\node at (0,-1) {$P$};
			\node at (1.25,-0.9) {$\scriptstyle{\tld{f}}$};
			\draw [->] (0.3,-1) -- (2.2,-1);
			\node at (2.5,-1) {$Q$};

			\draw [->] (0,-0.2) -- (0,-0.8);
			\draw [->] (2.5, -0.2) -- (2.5, -0.8);
			\node at (0.2,-0.5) {$\scriptstyle{h_1}$};
			\node at (2.7,-0.5) {$\scriptstyle{h_2}$};
		\end{tikzpicture}
	\end{center}

	Mather's theorem is exposited in \cite{gibson}.
\end{defin*}

When $d-k$ is small there are natural choices for $P$. Here we mention a few; as $d-k$ grows the choices proliferate.

\begin{table}[!ht]
	\centering
	\begin{tabular}{ll}
	$(d-k) = \dim \pi^{-1}(x)$ & Property $(P)$ \\
	\hline
	0 & \# points \\
	1 & \# of components (circles) \\
	2 & \# components, $g=$genius, $\sum$ genus (component) + \# $S^2$'s, $\dots$ \\
	3 & \# components, Gromov-Thurston volume \\
	4 & \# components, Gromov-Thurston volume, signature \\
	$\ast$ & Euler char. and semi-characteristic may be interesting choices
	\end{tabular}
\end{table}

(Actually if the condition that $x \in \R^k$ be a regular value is dropped, in the range $-d < d-k < 0$, $P = \#$ points is still interesting: $M^d$ has $k$-wdith$(M^d) = 1$ iff $M$ embeds in $\R^k$.)

Here are (consequences of) a few familiar theorems in this notation:

\vspace{0.5em} \noindent \textbf{Montesinos \cite{montesinos}}. All 3-manifolds are 3-fold branched covers of $S^3 \implies$
\begin{equation}
	\mathrm{3-width}(M^3) \leq 6
\end{equation}
(Pf: Use the composition $M \xrightarrow{\text{3-1 b.c.}}S^3 \xrightarrow{\text{2-1 folding}} \R^3$. This composition is not actually a topologically stable map because the local model normal to the branch locus resolves into folds and cusps, but one may perturb this composition to achieve a topologically stable map without increasing the number of (regular) point preimages.)

\vspace{0.5em} \noindent \textbf{Piergallini \cite{piergallini}}. All 4-manifolds are 4-fold branches covers of $S^4 \implies$
\begin{equation}
	\mathrm{4-width}(M^4) \leq 8
\end{equation}
(Similar proof)

\vspace{0.5em} Some obvious facts are
\begin{align}
	& \text{1-width}(M^1) = \text{1-width}(\coprod_{i=1}^k S^1) = 2, \\
	& \text{1-width}(M^2) \leq 2
\end{align}
with equality except for $S^2$ and the non-orientable surfaces where 1-width is 1.
\begin{equation}
	\text{If } P = \text{genus, 1-width}_P(M^3) \text{ is an unbounded function}
\end{equation}

That is, there is a sequence of 3-manifolds $\{M_i\}$ on which the minimax genus of a "sweept out" by surface diverges. (Pf: From Rubenstein-Pitts \cite{rp}, and later work e.g. \cite{ketover} it is known the minimax sweep outs lead to embedded minimal surfaces $\Sigma$ of no larger genus than the genus of surfaces in the sweep out. If $\{M^3_i\}$ are hyperbolic with diverging injectivity readius area$(\Sigma)$ also diverges. But for a minimal surface in a hyperbolic manifold
\[
-\chi(\Sigma) \geq \frac{\text{Area}(\Sigma)}{2\pi}
\]

Possibly similar arguments could show that for $P=$ Gromov-Thurston volume that $k$-$\text{width}_P(M^d)$ diverges for $d-k \geq 2$.)

Finally, \cite{groupw} answered a math overflow question. Using the current terminology, it proved that if $P$ is "rank of $H_1(;\Z)$" then
\begin{equation}
	\text{1-width}_P(M^d) \text{ diverges for all } d \geq 4
\end{equation}

To this list we add one additional theorem on width of manifolds:

\begin{thm}
	For all 3-manifolds $M$, 2-width$(M) \leq 2$. If 2-width$(M) = 1$ then $M = S^3$ or $(\underset{k}{\#}(S^1 \times S^3) \# (\underset{j}{\#} S^2 \tld{\times} S^2))$. That is, every 3-manifold admits a map to $\R^2$ with all regular value preimages consisting of one or two circles. If 2 circles \emph{never} occur then $M$ is one of the very special manifolds above.
\end{thm}

The proof will be given in section 3.

Now turn to $k$-width of an embedding $M^d \overset{e}{\hookrightarrow} \R^n$. Again given a numerical property $P$,
\begin{defin*}
	$k$-width$(e) = \mathrm{minimax}\ P((\pi \circ e)^{-1}(x))$, where the minimum is over all embeddings isotopic to $e$ (still denoted $e$) while the maximum is over all $x$, regular values for $e \circ \pi$, $\pi: \R^n \ra \R^k$ projection onto the last $k$ coordinates.
\end{defin*}
One could consider the case where $P$ depends not just on the manifold $(\pi \circ e)^{-1}(x)$ but also on its embedding in the fiber of $\pi$, $\R^{n-k}$, but we do \emph{not} do this here.

The main result of this paper is:

\begin{thm}
	There is a family of embeddings, $e_i: T^3 \hookrightarrow \R^4$, of the 3-torus into $\R^4$ for which 2-width$(e_i)$ diverges. That is, as $i$ increases the number of circles which \emph{must} occur in some preimage $(\pi \circ e)^{-1}(x)$ for some regular $x \in \R^2$ diverges.
\end{thm}

The proof will be given in section 2.

Notice the contrast with Theorem 1; it is the embedding, not the source manifold, that forces the preimage to increase with complexity.

The paper is organized as follows:
\begin{itemize}
	\item Section 2 proves Theorem 2.
	\item Section 3 proves Theorem 1.
	\item Section 4 explains how 2-width---in light of Theorem 2---\emph{might} be a tool for proving the higher genus Georitz groups infinitely generated.
\end{itemize}

To say just a word about section 4 here, recall that the Goeritz group $G_g$ can be defined as the "motion group,"
\begin{equation}
	G_g = \pi_1 \text{ (space of genus $g$ Heegaard surfaces in $S^3$)}
\end{equation}

It has long been known that 5 "generators" proposed by Powell [P] do in fact generate $G_2$ and it was proved in \cite{dehn} that these five also generate $G_3$. It is currently an open question whether $G_g$, $g \geq 4$, are finitely or infinitely generated. We show that $G_g$ is infinitely generated iff there are motions (loops) of the Heegaard surface $\Sigma_g$ in which the "complexity" of some individual "frame" of the movie (i.e. motion) diverges. The movie traces out an embedding of $e_{\text{movie}}: M_{\text{movie}} \hookrightarrow \R^4$, the mapping torus 3-manifold. We show that 2-width$(e_{\text{movie}})$ lower bounds "frame complexity," so if the former diverges over motions of $\Sigma_g$, then $G_g$ must be infinitely generated.

The final topic is really just a comment connecting two difficult subjects, and can be made here. It is well known that the (orientable) surface of genus $g$ is obtained by 2-fold branched covers of the 2-sphere, $S^2$ along $2g+2$ point. If those points are allowed to move (without collision) in loop, that motions defines a map from the spherical braid group $B^s_{2g+2}$ to the mapping class group of the 2-fold branched cover
\begin{equation}
	\te: B^s_{2g+2} \ra \mathrm{MCG}(\Sigma_g)
\end{equation}

Exactly the same construction can be done in four dimension where 2-fold branched cover about a standardly embedded surface of genus $g$ $\Sigma_g$ yields $\underset{g}{\#} S^2 \times S^2$. On the level of motion groups and mapping class groups, we obtain
\begin{equation}\label{arrows}
	\begin{split}
		G_g = \pi_1 \text{ (genus $g$ He} & \text{egaard surfaces in $S^3$)} \\
		& \big\downarrow \\
		\pi_1 \text{ (smoothly standard} & \text{ embeddings of $\Sigma_g$ in $S^4$)} \\
		& \big\downarrow \\
		\mathrm{MCG}(\underset{g}{\#} S^2 \times S^2) & = \pi_0(\mathrm{Diff} \underset{g}{\#} S^2 \times S^2)
	\end{split}
\end{equation}

The first arrow in (\ref{arrows}) is defined by including $S^3$ as the equator of $S^4$. The second arrow is defined by tracking the sheets of the branched cover as the branching locus moves. Just as in the two dimensional case, this tracking covers the motion by a family of diffeomorphisms from the initial to subsequence branched covers. When the loop closes a self-diffeomorphism is defined.

Little is known about any of the three groups in (\ref{arrows}). Since one motivation of the present paper is to develop a tool for proving $G_g$ is infinitely generated for $g \geq 4$, we note that this same tool may be relevant for the study of the presently mysterious four dimensional mapping class groups.

\makeatletter
\renewcommand{\thethm}{\thesection.\arabic{thm}}
\@addtoreset{thm}{section}
\makeatother

\section{2-width$(e_i: T^3 \hookrightarrow \R^4)$}

The starting point for the construction of $e_i:T^3 \hookrightarrow \R^4$, $i = 1,2,3,\dots$ is the example of Litherland and Asano (\cite{litherman}, \cite{asano}) of a smoothly embedded torus $f_1: T^2 \hookrightarrow \R^4$ with the property that the inclusion of the peripheral 3-torus, $T^3 \overset{e_i}{\hookrightarrow} \R^4 \setminus f_1(T^2)\ \eqqcolon\ Y$ induces an injection on $\pi_1$,
\begin{equation}
	\{0\} \ra \pi_1(T^3) \xrightarrow{e_{i\#}} \pi_1(Y)
\end{equation}

Consistent with the terminology of \cite{fh} we say $Y$ is \emph{homologically rich} since $Y$ has a covering space $\tld{Y}$ with cup products (we use integer coefficients) of the maximum possible length, three; $\pi_1(\tld{Y}) = e_{1\#}(\pi_1(T^3))$, i.e. $\tld{Y}$ is the peripheral cover. Denote the closed tubular neighborhood of $e(T^2)$ by $X$, so $\R^4 = X \cup Y$. Let $d: \R^4 \ra \R^4$ be any diffeomorphism and $\pi$ be the composition $\R^4 \xrightarrow{d} \R^4 \ra \R^2$ of $d$ followed by projection onto the last two coordinates.

\begin{lemma}\label{old21}
	For some $p \in \R^2$, a regular value of $\pi \circ e_1: T^3 \ra \R^2$, $\pi^{-1}(p) \cap X$ the compact planar domain $D$ within the Fiber $\R^2_p = \pi^{-1}(p)$, will have at least two connected components $D_1$ and $D_2$ of $D$ so that
	\[
	\mathrm{inc}_\ast[D_j, \de] \neq 0 \in H_2(X, \de;\Z), \ j=1,2
	\]
\end{lemma}

We postpone the proof momentarily to the show how Theorem 2 follows.

\begin{proof}[Proof of Theorem 2]
	$X$ is (diffeomorphic to) a product $T^2 \times D^2$ and choose an identification. There are (many) ways $k_i$ to re-embed $T^2 \times 0 \overset{k_1}{\hookrightarrow}T^2 \times D^2$ the torus into $T^2 \times D^2$ so that the composition, call it $K_i$
	\begin{equation}
		T^2 \xrightarrow{k_i} T^2 \times D^2 \xrightarrow{\text{proj.}} T^2
	\end{equation}
	is a degree $i$ covering space projection with the property that $\pi_1(T^2) / K_{i\#}(\pi_1(T_2)) \cong \Z / i\Z \coloneqq \Z_i$. For example, if the $D^2$ is given polar coordinates $(\rho, \te_0)$, $0 \leq \rho \leq 1$, and $T^2$ is parametrized by $(\te_1, \te_2)$, then one may take $(\te_1, \te_2) \ra (i\te_1, \te_2, \frac{1}{2}, \te_1)$. The composition
	\begin{equation}
		f_i: T^2 \overset{k_1}{\hookrightarrow} X \xrightarrow{\mathrm{inc}_X} \R^4
	\end{equation}
	can be call an \emph{$i$th cabling} of $f_i$.

	Let $X_i \subset X$ be a closed tubular neighborhood of $f_i(T^2)$ and let $e_i$ denote the inclusion of $\de X_i \cong T^3$ into $\R^4$.

	Now consider one of the essential domains. say $D_1 \subset \pi^{-1}(p)$, and look within $D_1$ at $E \coloneqq \pi^{-1}(p) \cap X_i$. By perturbing $p \in \R^2$, if necessary we may assume that $p$ is both a regular value for $e_1$ and $e_i$ so $E_i \subset D$ is a compact sub-surface.

	\begin{lemma}\label{old22}
		There are at least $i$ components of $E_i$, call them $F_1, \dots, F_i$ so that $[F_i, \de] \neq 0 \in H_2 (X_i, \de;\Z)$.
	\end{lemma}

	\begin{proof}
		Further assuming $p$ is a regular value for the composition $T^2 \xrightarrow{f_i} X \xrightarrow{\pi} \R^2$, the number of inverse images of $p$, counted according to sign, must be $i$. In fact, since $\pi_1(T^2) / \mathrm{inc_{\#}}\pi_1(T^2) \cong \Z_i$ we may choose a simple closed curve (scc) downstairs $\gamma \subset T^2$ so that its inverse image $\lbar{\gamma}$, $\pi \circ f_i: \lbar{\gamma} \ra \gamma$, contains a scc $\tld{\gamma} \subset \lbar{\gamma}$ so that $\pi \circ f_i \big\vert_{\tld{\gamma}}: \tld{\gamma} \ra \gamma$ is homotopic to an $i$-fold covering map. Among the inverse images of $p$ there are $i$ points, $p_1, \dots, p_i$ with the property that the arc segment from $p_r$ to $p_s$, $0 \leq r < s \leq i$ maps to $\gamma$ degree $(s-r)$. Suppose, for a contradiction, $p_r, p_s \subset F_k$, $1 \leq r < s \leq i$, lie in the same $F_k$, $1 \leq k \leq i$.

		Then choose a "shortcut" arc $\alpha \subset F_k$ joining $p_r$ to $p_s$. Then the union $\delta$ of $\alpha$ with the (oriented) segment $\beta$ of $\tld{\gamma}$ running from $p_s$ to $p_r$ maps with degree $(s-r)$ to $\gamma$, contradicting $\pi_1(T^2) / K_{i\#} \pi_1(T^2) \cong \Z_i$, since $(s-r) \not\equiv 0$ mod $i$. Thus each $p_s$, $1 \leq s \leq i$, must lie in a distinct component, call them $F_1, \dots, F_i$ of $E_i$, proving the lemma.
	\end{proof}

	From here Theorem 2 is immediate, $\de E_i$ must contain at least $i$ scc, at least one belonging to $\de F_k$, $1 \leq k \leq i$.
\end{proof}

\begin{proof}[Proof of Lemma \ref{old21}]
	As a whole $[D] = 0 \in H_2(X, \de ; \Z)$. So if there is one essential component $D_1$ there must be others to cancel its class. To obtain a contradiction we therefore assume all $[D_i] = 0$.

	Consider $\pi^{-1}(p) \cap X \subset \pi^{-1}(p)$; picture it as a black compact subsurface in an otherwise white plane. An innermost circle which bounds a white disk must be null homotopic in $e_1(T^3)$ since the peripheral group injects into $\pi_1(Y)$, the group of the complement. An innermost circle bounding a black disk might be trivial, but in any case is an element $w \in \ker(\pi_1(\de X) \ra \pi_1(X)) \cong \Z$, generated by the meridional circle. (This situation is exceptionally simple since $\pi_1(T^3) \cong \pi_1 \de X$ is abelian.) It follows from our assumption that for each component $D_i$, $[D_i] = 0 \in H_2 (X, \de; \Z)$ that $w = 0$ (for each circle bounding a black disk) since $w$ is in the image of the boundary map
	\begin{equation}\label{boundary}
		H_2(X, \de;\Z) \xrightarrow{\de} H_1(\de X; \Z)
	\end{equation}

	This implies that all next-to-innermost circles are trivial in $\pi_1(\de X)$. There are two cases: if the color immediately to the inside of such a circle is white then any black spots within can be "cut off" by a null homotopy near $\de X$, staying on the white (i.e. $Y$) side, using the fact that $w = 0$. Similarly, if the color immediately to the inside is black, since any white spots must have trivial boundary from the injectivity of the peripheral subgroup. Now our homological assumption (that the lemma is false) together with line (\ref{boundary}) shows that the outer boundary also must be trivial, continuing in this way (inducting on the depth of the nesting pattern of the circles in $\pi^{-1}(p) \cap \de X$) we conclude that \emph{all} such circles are trivial in $\pi_1(\de X) \cong \pi_1(T^3) \cong \Z^3$. We have established the following:

	\begin{claim}
		If Lemma \ref{old21} is false then for every regular vlaue $p$ all the scc of $\pi^{-1}(p) \cap \de X$, which of course are precisely the scc of $\pi^{-1}(p) \cap \de Y$ since $\de X = \de Y$, are trivial in $\pi_1(\de X) \cong \pi_1(\de Y) \cong \pi_1 T^3 \cong \Z^3$.
	\end{claim}

	We now explore, in the spirit of \cite{fh}, the homological implication of such \emph{triviality} and find a contradiction. For each regular $p$, $Y_p \coloneqq \pi^{-1}(p) \cap Y$ is a planar domain of finite type and hence homotopy equivalent to a finite 1-complex $Y_p$. $\pi_1(Y_p)$ is normally generated by its boundary scc so triviality implies that the inclusion $Y_p \subset Y$ is homotopically trivial.

	Everything of interest lies within $B_r$, the ball of radius $r$ around the origin of $\R^4$. So, with changing notations, let us restrict all maps to that ball and replace $Y$ with $Y \cap B_r$. In this way everything is compact.

	Thus, for each regular $p$, $Y_p$ has a regular neighborhood $\mathcal{N}(Y_p)$ so that $\mathcal{N}(Y_p) \ra Y$ is also homotopically trivial. So for each regular $p$ there is an open neighborhood $U_p$ of $p$ so that for all $p^\pr \in U_p$, $Y_{p^\pr} \subset \mathcal{N}(Y_p)$. $\{U_p, p \text{ regular}\}$ provides an open cover of $B_r^2$, the ball of radius $r$ in $\R^2$. By compactness there is a finite subcover $\{U_q\}$, $q$ belonging to a finite set of regular values. For each $q$, the preimages $P_q \coloneqq \pi^{-1}(U_q) \cap Y$ are homotopically trivial in $Y$, since they factor through the neighborhood $\mathcal{N}(Y_p)$.

	Now give $B_r^2$ a fine smooth handle decomposition $\mathcal{H}$ with the union of $i$-handles denoted $h_i$. Fine means that each $i$-handle of $\mathcal{H}$ lies in some $U_q$, for $i = 0,1,2$. Thus fineness guarantees that $Y_i \coloneqq \pi^{-1}(h_i) \cap Y \subset Y$ is null homotopic in $Y$, $i = 0,1,2$.

	Let $\tld{Y}$ denote the cover of $Y$ associated with the peripheral subgroup $\pi_1(\de Y) \cong \Z^3$. Because each $Y_i$, $0 \leq i \leq 2$ is null homotopic in $Y$, each component of each $Y_i$ lifts to $\tld{Y}$ (in $\pi_1 (\de Y)$ ways). Let $\tld{Y}_i$ denote the preimage of $Y_i$ in $\tld{Y}$; it is a union of $\pi_1(\de Y)$-many copies of $Y_i$. Each component of $\tld{Y}_i$ is null homotopic in $\tld{Y}$. In particular, the $\de$ map below is onto.
	\begin{equation}\label{de-map}
		H_2(\tld{Y}, \tld{Y}_i; \Z) \xrightarrow{\de} H_1(\tld{Y}_i; \Z) \xrightarrow{0} H_1(\tld{Y}; \Z)
	\end{equation}
	since the next map in the exact sequence is zero.

	The columns below are exact sequences from the universal coefficient theorem and the map $\delta$ being the hom dual of $\de$ on line (\ref{de-map}) is an injection.

	\begin{center}
	\begin{tikzpicture}
		\node at (-2,3) {0};
		\node at (2, 3) {0};
		\node at (-2,2.5) {$\downarrow$};
		\node at (2,2.5) {$\downarrow$};
		\node at (-2,2) {$\mathrm{Ext}_Z^1(H_0(\tld{Y}_i;\Z),\Z)$};
		\node at (2,2) {$\mathrm{Ext}_Z^1(H_1(\tld{Y}, \tld{Y}_i;\Z),\Z)$};
		\node at (-2,1.5) {$\downarrow$};
		\node at (2,1.5) {$\downarrow$};
		\node at (-2,1) {$H^1(\tld{Y}_i;\Z)$};
		\node at (2,1) {$H^2(\tld{Y},\tld{Y}_i;\Z)$};
		\node at (0,1.1) {$\overset{\delta}{\longrightarrow}$};
		\node at (-2, 0.5) {$\downarrow$};
		\node at (2, 0.5) {$\downarrow$};
		\node at (-2,0) {$\mathrm{Hom}(H_1(\tld{Y}_i;\Z),\Z)$};
		\node at (2.2,0) {$\mathrm{Hom}(H_2(\tld{Y},Y_i;\Z),\Z)$};
		\node at (0,0.1) {$\overset{\de^\ast}{\longrightarrow}$};
		\node at (-2, -0.5) {$\downarrow$};
		\node at (2, -0.5) {$\downarrow$};
		\node at (-2,-1) {0};
		\node at (2,-1) {0};
		\node at (-4,0) {$\ra$};
		\node at (-4.5,0) {0};
	\end{tikzpicture}
	\end{center}

	The upper left Ext vanishes since $H_0$ is torsion free, implying that the coboundary $\delta$ is also an injection. Thus the cohomology sequence of the pair $(\tld{Y}, \tld{Y}_i)$ factors through the indicated zero.

	\begin{center}
	\begin{tikzpicture}
		\node at (0,0) {$H^1(\tld{Y}, \tld{Y}_i;\Z) \rightarrow H^1(\tld{Y};\Z) \rightarrow H^1(\tld{Y}_i;\Z) \overset{\delta}{\rightarrow} H^2(\tld{Y}, \tld{Y}_i;\Z)$};
		\draw [->] (-0.9, 0.3) to (-0.2, 0.6);
		\draw [->] (0.2, 0.6) to (0.9, 0.3);
		\node at (0,0.6) {0};
	\end{tikzpicture}
	\end{center}

	This means that the first arrow is onto, permitting us to make a Lusternik-Schnirelmann style argument. Let $\alpha$, $\beta$, and $\gamma \in H^1(\tld{Y};\Z)$. Choose cochain representatives $\lbar{\alpha}$ for $\alpha$ vanishing on $\tld{Y}_0$ (i.e. the image of a representative from the leftmost module $H^1(\tld{Y}, Y_0; \Z)$), $\lbar{\beta}$ for $\beta$ vanishing on $\tld{Y}_1$, and $\lbar{\gamma}$ for $\gamma$ vanishing on $\tld{Y}_2$. We may compute $\alpha \cup \beta \cup \gamma \in H^3(\tld{Y};\Z)$ using any representatives we like. Using $\lbar{\alpha}$, $\lbar{\beta}$, $\lbar{\gamma}$ we find that $\alpha \cup \beta \cup \gamma$ lies in the image of
	\[
	H^3(\tld{Y}, \bigcup_{i=0}^2 \tld{Y}_i; \Z) \ra H^3(\tld{Y};\Z)
	\]
	but $\cup_{i=0}^2 \tld{Y}_i = \tld{Y}$ so the first module vanishes and $\alpha \cup \beta \cup \gamma = 0 \in H^3(\tld{Y}_i;\Z)$.

	By construction, the inclusion $\de \tld{Y} \overset{\text{inc.}}{\hookrightarrow} \tld{Y}$ induces an isomorphism on fundamental groups so the only obstruction to constructing a retraction $r: \tld{Y} \ra \de \tld{Y}$, $r \circ \text{inc.} = \id_{\de \tld{Y}}$, lies in $H^3(\tld{Y}; \pi_2(\de \tld{Y}))$. But $\de \tld{Y} \cong T^3$, a $K(\pi,1)$ so the obstruction vanishes and the retraction $r$ exists. Let $a$, $b$, $c \in H^1 (\de \tld{Y};\Z)$ be given by the respective projections to the three cirlce factors of $T^3 \cong \de \tld{Y}$,
	\begin{equation}
		a \cup b \cup c = 1 \in H^3(\de \tld{Y}; \Z) \cong \Z
	\end{equation}

	Now set $\alpha = r^\ast(a)$, $\beta = r^\ast(b)$, and $\gamma = r^\ast(c)$. We have $\mathrm{inc}^\ast(\alpha) = a$, $\mathrm{inc}^\ast(\beta) = b$, and $\mathrm{inc}^\ast(\gamma) = c$, so $\mathrm{inc}^\ast(\alpha \cup \beta \cup \gamma) = \mathrm{inc}^\ast(0)$. So in $H^3(\de \tld{Y}, \Z)$ we find
	\begin{equation}
		0 = \mathrm{inc}^\ast(0) = \mathrm{inc}^\ast(\alpha \cup \beta \cup \gamma) = \mathrm{inc}^\ast(\alpha) \cup \mathrm{inc}^\ast(\beta) \cup \mathrm{inc}^\ast(\gamma) = a \cup b \cup c = 1
	\end{equation}

	This contradiction proves Lemma \ref{old21}.
\end{proof}

\section{2-width$(M^3)$}

We prove Theorem 1. Let $M$ be a closed 3-manifold and $f:M \ra \R$ a Morse function. The regular levels are surfaces which we denote by $\Sigma_x = f^{-1}(x)$. Now lift $f$ to $g$ to obtain a commutative diagram

\begin{center}
	\begin{tikzpicture}[scale=1.3]
		\node at (0,0) {$M$};
		\node at (0.8,0.1) {$\xrightarrow{\makebox[1cm]g}$};
		\node at (1.7,0) {$\mathbb{R}^2$};
		\node at (2.6,0.1) {$\xrightarrow{\makebox[1cm]{$\pi_2$}}$};
		\node at (3.5,0) {$\mathbb{R}^1$};
		\node at (1.7,-0.6) {$\Big\downarrow$};
		\node at (1.7,-1.2) {$\mathbb{R}^1$};
		\node at (1.95,-0.65) {$\pi_1$};
		\draw [->] (0.2, -0.2) -- (1.4, -1.1);
		\node at (0.6,-0.8) {$f$};
		\node at (2.75,-0.35) {$(x,y) \mapsto y$};
		\node [rotate=270] at (2.4,-0.8) {$\mapsto$};
		\node at (2.4,-1.1) {$x$};
	\end{tikzpicture}
\end{center}

Clearly if $(x,y) \in \R^2$ is a regular value of $g$ then $\pi_1(x,y) = x$ is a regular value of $f$. Generically $g$ will have the property that for all regular $x$, $g \big\vert_{\Sigma_x}$ is Morse except for isolated $\{x_i\}$ where $g \big\vert_{\Sigma_x}$ has a generic Cerf transition (birth/death of canceling pair). What we do in this section is construct $g$ very carefully so that $g^{-1}$(regular$(x,y)$) consists of at most two circles.

Around each critical point $x_c \in \R$ of $f$ we will require $g$ to assume a canonical form, illustrated in Figure 3.1 by drawing the required level sets of $g$ near the $\Sigma_{x_c}$.

\begin{figure}[!ht]
	\centering
	\begin{tikzpicture}[scale=0.35]
		\draw (0,0) circle (3);
		\draw [dashed] (3,0) arc (0:180:3 and 0.6);
		\draw (3,0) arc (0:-180:3 and 0.6);
		\draw [dashed] (1.8, -2.4) arc (0:180:1.8 and 0.3);
		\draw (1.8, -2.4) arc (0:-180:1.8 and 0.2);
		\draw [dashed] (1.8, 2.4) arc (0:180:1.8 and 0.2);
		\draw (1.8, 2.4) arc (0:-180:1.8 and 0.3);

		\draw [dashed] (0, -3.2) to (0, -4.9);
		\draw (0, -4.9) ellipse (1.5 and 0.5);
		\draw (0.2, -6.5) arc (100:260:0.5 and 0.75);
		\draw (0.02, -6.6) arc (80:-80:0.3 and 0.65);
		\draw (-1.48, -4.8) arc (120:240:1.4 and 2.8);
		\draw (1.48, -4.8) arc (60:-60:1.4 and 2.8);
		\draw (-1.48, -9.65) .. controls (-1.1, -10.2) .. (-1.48, -10.8);
		\draw (0.2, -12.5) arc (100:260:0.5 and 0.75);
		\draw (0.02, -12.6) arc (80:-80:0.3 and 0.65);
		\draw (-1.48, -10.8) arc (120:240:1.4 and 2.8);
		\draw (1.48, -10.8) arc (60:-60:1.4 and 2.8);
		\draw (1.48, -9.65) .. controls (1.1, -10.2) .. (1.48, -10.8);
		\draw (-1.5, -15.55) arc (180:360:1.5 and 0.5);
		\draw [dashed] (-1.5, -15.55) arc (180:0:1.5 and 0.4);
		\draw (-2.18, -7.2) arc (180:360:0.98 and 0.3);
		\draw [dashed] (-2.18, -7.2) arc (180:0:0.98 and 0.3);
		\draw (0.27, -7.2) arc (180:360:0.95 and 0.3);
		\draw [dashed] (0.27, -7.2) arc (180:0:0.95 and 0.3);
		\draw (-2.18, -13.2) arc (180:360:0.98 and 0.3);
		\draw [dashed] (-2.18, -13.2) arc (180:0:0.98 and 0.3);
		\draw (0.27, -13.2) arc (180:360:0.95 and 0.3);
		\draw [dashed] (0.27, -13.2) arc (180:0:0.95 and 0.3);
		\draw (-1.63, -11) arc (180:360:1.63 and 0.4);
		\draw [dashed] (-1.63, -11) arc (180:0:1.63 and 0.35);
		\draw (-1.63, -9.4) arc (180:360:1.63 and 0.35);
		\draw [dashed] (-1.63, -9.4) arc (180:0:1.63 and 0.4);
		\draw [dashed] (0, -16.2) to (0, -17.7);

		\draw [dashed] (0, -18.2) to (0, -19.9);
		\draw (0, -19.9) ellipse (1.5 and 0.5);
		\draw (0.2, -21.5) arc (100:260:0.5 and 0.75);
		\draw (0.02, -21.6) arc (80:-80:0.3 and 0.65);
		\draw (-1.48, -19.8) arc (120:240:1.4 and 2.8);
		\draw (1.48, -19.8) arc (60:-60:1.4 and 2.8);
		\draw (-1.48, -24.65) .. controls (-1.1, -25.2) .. (-1.48, -25.8);
		\draw (0.2, -27.5) arc (100:260:0.5 and 0.75);
		\draw (0.02, -27.6) arc (80:-80:0.3 and 0.65);
		\draw (-1.48, -25.8) arc (120:240:1.4 and 2.8);
		\draw (1.48, -25.8) arc (60:-60:1.4 and 2.8);
		\draw (1.48, -24.65) .. controls (1.1, -25.2) .. (1.48, -25.8);
		\draw (-1.5, -30.55) arc (180:360:1.5 and 0.5);
		\draw [dashed] (-1.5, -30.55) arc (180:0:1.5 and 0.4);
		\draw (-2.18, -22.2) arc (180:360:0.98 and 0.3);
		\draw [dashed] (-2.18, -22.2) arc (180:0:0.98 and 0.3);
		\draw (0.27, -22.2) arc (180:360:0.95 and 0.3);
		\draw [dashed] (0.27, -22.2) arc (180:0:0.95 and 0.3);
		\draw (-2.18, -28.2) arc (180:360:0.98 and 0.3);
		\draw [dashed] (-2.18, -28.2) arc (180:0:0.98 and 0.3);
		\draw (0.27, -28.2) arc (180:360:0.95 and 0.3);
		\draw [dashed] (0.27, -28.2) arc (180:0:0.95 and 0.3);
		\draw (-1.63, -26) arc (180:360:1.63 and 0.4);
		\draw [dashed] (-1.63, -26) arc (180:0:1.63 and 0.35);
		\draw (-1.63, -24.4) arc (180:360:1.63 and 0.35);
		\draw [dashed] (-1.63, -24.4) arc (180:0:1.63 and 0.4);
		\draw [dashed] (0, -31.2) to (0, -32 .7);

		\node at (4.7,1) {$\xrightarrow{\text{index 3}}$};
		\node at (4.7,-1) {$\xleftarrow{\text{index 0}}$};
		\node at (9.2,0) {\Huge{$\varnothing$}};
		\node at (4.7,-2.5) {(a)};
		\node at (4.7,-9.4) {$\xrightarrow{\text{index 2}}$};
		\node at (4.7,-11) {$\xleftarrow{\text{index 1}}$};
		\node at (4.7,-15.3) {(b)};
		\node at (4.7,-24.4) {$\xrightarrow{\text{index 2}}$};
		\node at (4.7,-26) {$\xleftarrow{\text{index 1}}$};
		\node at (4.7,-30.3) {(c)};

		\draw (9.4,-27.5) arc (100:260:0.5 and 0.75);
		\draw (9.22,-27.6) arc (80:-80:0.3 and 0.65);
		\draw (7.72,-25.8) arc (120:240:1.4 and 2.8);
		\draw (10.68,-25.8) arc (60:-60:1.4 and 2.8);
		\draw (7.7,-30.55) arc (180:360:1.5 and 0.5);
		\draw [dashed] (7.7,-30.55) arc (180:0:1.5 and 0.4);
		\draw (7.02,-28.2) arc (180:360:0.98 and 0.3);
		\draw [dashed] (7.02,-28.2) arc (180:0:0.98 and 0.3);
		\draw (9.47,-28.2) arc (180:360:0.95 and 0.3);
		\draw [dashed] (9.47,-28.2) arc (180:0:0.95 and 0.3);
		\draw [dashed] (9.2,-31.2) to (9.2, -32 .7);
		\draw (9.2,-25.9) ellipse (1.5 and 0.5);

		\draw [dashed] (9.2, -18.2) to (9.2, -19.9);
		\draw (9.2, -19.9) ellipse (1.5 and 0.5);
		\draw (9.4, -21.5) arc (100:260:0.5 and 0.75);
		\draw (9.22, -21.6) arc (80:-80:0.3 and 0.65);
		\draw (9.2-1.48, -19.8) arc (120:240:1.4 and 2.8);
		\draw (10.68, -19.8) arc (60:-60:1.4 and 2.8);
		\draw (9.2-2.18, -22.2) arc (180:360:0.98 and 0.3);
		\draw [dashed] (9.2-2.18, -22.2) arc (180:0:0.98 and 0.3);
		\draw (9.47, -22.2) arc (180:360:0.95 and 0.3);
		\draw [dashed] (9.47, -22.2) arc (180:0:0.95 and 0.3);
		\draw (10.7, -24.62) arc (-10:-170:1.52 and 0.5);
		\draw [dashed] (10.7, -24.62) arc (10:170:1.52 and 0.5);

		\draw [dashed] (9.2, -3.2) to (9.2, -4.9);
		\draw (9.4, -12.5) arc (100:260:0.5 and 0.75);
		\draw (9.22, -12.6) arc (80:-80:0.3 and 0.65);
		\draw (9.2-1.48, -10.8) arc (120:240:1.4 and 2.8);
		\draw (10.68, -10.8) arc (60:-60:1.4 and 2.8);
		\draw (9.2-1.5, -15.55) arc (180:360:1.5 and 0.5);
		\draw [dashed] (9.2-1.5, -15.55) arc (180:0:1.5 and 0.4);
		\draw (9.2-2.18, -13.2) arc (180:360:0.98 and 0.3);
		\draw [dashed] (9.2-2.18, -13.2) arc (180:0:0.98 and 0.3);
		\draw (9.47, -13.2) arc (180:360:0.95 and 0.3);
		\draw [dashed] (9.47, -13.2) arc (180:0:0.95 and 0.3);
		\draw [dashed] (9.2, -16.2) to (9.2, -17.7);
		\draw (10.7, -10.8) arc (-10:-170:1.51 and 0.5);
		\draw (9.2, -4.9) ellipse (1.5 and 0.5);
		\draw (9.2-1.48, -4.8) arc (120:240:1.4 and 3.47);
		\draw (10.67, -4.8) to [out=-60, in=55] (11.1, -7.6) to [out=235, in=-55] (10.5, -7.6) to [out=125, in=0] (9.7, -6.6);
		\draw (9.7, -6.6) arc (90:270:0.6 and 1.57);
		\draw (9.7, -9.75) to [out=0, in=225] (10.3, -8.9) to [out=45, in=135] (11.1,-8.9) to[out=-45, in=50] (10.7, -10.8);
		\draw (9.17, -8.9) arc (0:-180:1.03 and 0.3);
		\draw [dashed] (9.17, -8.9) arc (0:180:1.03 and 0.3);
		\draw (9.17, -7.5) arc (0:-180:1.07 and 0.3);
		\draw [dashed] (9.17, -7.5) arc (0:180:1.07 and 0.3);
		\draw (11.32, -9.3) arc (0:-180:0.61 and 0.15);
		\draw [dashed] (11.32, -9.3) arc (0:180:0.61 and 0.15);
		\draw (11.3, -7.2) arc (0:-180:0.5 and 0.2);
		\draw [dashed] (11.3, -7.2) arc (0:180:0.5 and 0.2);

		\draw [dashed] (18.4, -3.2) to (18.4, -4.9);
		\draw (18.4, -4.9) ellipse (1.5 and 0.5);
		\draw (16.9, -4.9) .. controls (17.9, -9.1) .. (18.4-1.48, -10.8);
		\draw (18.6, -12.5) arc (100:260:0.5 and 0.75);
		\draw (18.42, -12.6) arc (80:-80:0.3 and 0.65);
		\draw (18.4-1.48, -10.8) arc (120:240:1.4 and 2.8);
		\draw (19.88, -10.8) arc (60:-60:1.4 and 2.8);
		\draw (19.9,-4.9) .. controls (18.9,-9.1) .. (19.88, -10.8);
		\draw (18.4-1.5, -15.55) arc (180:360:1.5 and 0.5);
		\draw [dashed] (18.4-1.5, -15.55) arc (180:0:1.5 and 0.4);
		\draw (18.4-1.63, -11) arc (180:360:1.63 and 0.4);
		\draw [dashed] (18.4-1.63, -11) arc (180:0:1.63 and 0.35);
		\draw [dashed] (18.4-2.18, -13.2) arc (180:0:0.98 and 0.3);
		\draw (18.67, -13.2) arc (180:360:0.95 and 0.3);
		\draw [dashed] (18.67, -13.2) arc (180:0:0.95 and 0.3);
		\draw (18.4-2.18, -13.2) arc (180:360:0.98 and 0.3);
		\draw (19.16, -9.1) arc (0:-180:0.76 and 0.25);
		\draw [dashed] (19.16, -9.1) arc (0:180:0.76 and 0.25);
		\draw [dashed] (18.4, -16.2) to (18.4, -17.7);
		\draw (19.43, -6.9) arc (0:-180:1.03 and 0.25);
		\draw [dashed] (19.43, -6.9) arc (0:180:1.03 and 0.25);
		\draw [->] (12.9, -10.2) to (14.9, -10.2);
	\end{tikzpicture}
	\caption{}\label{levelsets}
\end{figure}

Note that in the immediate r.h.s. of the \ref{levelsets}(b) transition $g \big\vert_{\Sigma_x}$ can be simplified via two Morse cancellations. These should be implemented by suitable choice of $g$ on the lower-genus side (r.h.s.) of the transition.

The lift $g$ must be carefully chosen in the regions $R_i \subset \R$ between critical points, $R_i = f^{-1}(x_{e_i}, x_{e_{i+1}})$ to preserve the two-loop property (tlp). We explain how to do this next.

The 2D mapping class groups (MCG) are generated by a finite list of Dehn twists about the simple closed curves (scc) $\gamma$ shown in Figure \ref{twists} in reference to the standard height function. It suffices to explain how $g$ may be constructed on $\Sigma \times [0,1]$, to interpolate between the \emph{before} and \emph{after} height function $h \coloneqq h_0$ and $h_1 = D_\gamma \circ h$, where $D_\gamma$ is Dehn twist about $\gamma$, while retaining the tlp. $h$ is $g \big\vert_{\Sigma_x}$, at a generic value of $x$, far from any $x_{\text{critical}}$. We assume $h$ is a standard Morse function of $\Sigma_x$ with no Morse-cancelling handles present.

\begin{figure}[!ht]
	\centering
	\begin{tikzpicture}[scale=0.6]
		\draw (-1.5,0) arc (180:0:1.5 and 1.8);
		\draw (-1.5,0) to [out=-90, in=90] (-1.2,-1.7) to [out=-90,in=90] (-2.4,-4.4) to [out=-90, in=90] (-1.2,-7.1) to [out=-90,in=90] (-2.4,-9.8) to[out=-90,in=90] (-1.2,-12.5);
		\draw (1.5,0) to [out=-90, in=90] (1.2,-1.7) to [out=-90,in=90] (2.4,-4.4) to [out=-90, in=90] (1.2,-7.1) to [out=-90,in=90] (2.4,-9.8) to[out=-90,in=90] (1.2,-12.5);

		\draw (0.1, -3.6) arc (100:260:0.6 and 1);
		\draw (-0.09, -3.7) arc (80:-80:0.5 and 0.9);
		\draw (0.1, -9) arc (100:260:0.6 and 1);
		\draw (-0.09, -9.1) arc (80:-80: 0.5 and 0.9);
		\draw (0.1, 1) arc (100:260:0.6 and 1);
		\draw (-0.09, 0.9) arc (80:-80: 0.5 and 0.9);

		\draw (1.2, -1.7) arc (0:-180:1.2 and 0.25);
		\draw [dashed] (1.2, -1.7) arc (0:180:1.2 and 0.25);
		\draw (1.2, -7.1) arc (0:-180:1.2 and 0.25);
		\draw [dashed] (1.2, -7.1) arc (0:180:1.2 and 0.25);
		\draw (-0.09, -0.87) arc (90:-90:0.2 and 1.42);
		\draw [dashed] (-0.09, -0.87) arc (90:270:0.2 and 1.42);
		\draw (-0.09, -5.47) arc (90:-90:0.2 and 1.82);
		\draw [dashed] (-0.09, -5.47) arc (90:270:0.2 and 1.82);
		\draw (-0.09, -10.87) arc (90:5:0.2 and 1.82);
		\draw [dashed] (-0.09, -10.87) arc (90:175:0.2 and 1.82);

		\draw (-0.05, 1.8) arc (90:-90:0.1 and 0.43);
		\draw [dashed] (-0.05, 1.8) arc (90:270:0.1 and 0.43);
		\draw (-0.05,-0.1) ellipse (0.75 and 1.6);
		\draw (-0.05, -4.6) ellipse (1 and 1.75);

		\draw (2.4, -4.4) arc (0:-180:1.04 and 0.25);
		\draw [dashed] (2.4, -4.4) arc (0:180:1.04 and 0.25);

		\node at (5.5,-1.7) {\Large{$\gamma$'s}};
		\draw [->] (4.9, -1.5) to (0.1, 1.2);
		\draw [->] (4.9, -1.6) to (0.3, -0.6);
		\draw [->] (4.9, -1.7) to (1.1, -1.9);
		\draw [->] (4.9, -1.8) to (0.2, -2.5);
		\draw [->] (4.9, -1.9) to (0.8, -3.4);
		\draw [->] (4.9, -2) to (1.7, -4);

		\node [rotate=-30] at (2.6,0) {\tiny{type a}};
		\node [rotate=-15] at (2.2,-0.85) {\tiny{type b}};
		\node [rotate=5] at (2.1,-1.7) {\tiny{type c}};
		\node [rotate=8] at (1.4,-2.2) {\tiny{type d}};
		\node [rotate=20] at (2,-2.7) {\tiny{type b}};
		\node [rotate=35] at (2.4,-3.3) {\tiny{type e}};
		\node at (6.4,-4.3) {(one such assymmetric $\gamma$};
		\node at (6.95,-5) {is needed to generate MCG($\Sigma$))};
		\draw [->] (6.4,-3.7) to [out=90,in=-45] (2.9, -3.3);
		\node at (-10, 0) {\phantom{abcdef}};
	\end{tikzpicture}
	\caption{}\label{twists}
\end{figure}

The interpolation $\lbar{h}$ between $h_0$ and $h_1$ is constant when $\gamma$ is horizontal (types c and e); $\lbar{h} = h_0 \times \id = h_1 \times \id$. When $\gamma$ is vertical (types a, b, and d) we describe $\lbar{h}$ explicitly by drawing the level set of $\lbar{h} = \{h_t, 0 \leq t \leq 1\}$. We do this in Figure \ref{donut} and simultaneously count for the different level and different times $t$ how many global circles the levels we specify produce (which depends on boundary conditions external to the annulus $A_t$ in which the levels are drawn.) Before we draw the levels, note that there are only really two cases for $\gamma$. Types a and d can be displaced slightly so that the induced heigh function in a small annular neighborhood has the same topological level set pattern (Figure \ref{displace}) we see for type b:

\begin{figure}[!ht]
	\centering
	\begin{tikzpicture}
		\draw[pattern=horizontal lines] (0,0) circle [radius=1];
		\fill[color=white] (0,0) circle[radius=0.35];
		\draw (0,0) circle(0.35);
	\end{tikzpicture}
	\caption{}\label{donut}
\end{figure}

As far as global gluing is concerned (the boundary conditions) the level lines in type a and d are always glue up to a single circle whereas for type b the level sets (top to bottom) glue up to form first one circle, then two circles, then one circle. Figure \ref{displace} shows the $t$-evolution of level sets of $h_t$ on the annulus $A_t$, and how these complete globally (dashed lines) to closed level circles at generic $t \in [0,1]$.

\begin{figure}[!ht]
	\centering
	\begin{tikzpicture}
		\draw (0,0) circle (1.25) circle (0.25);
		\draw (-1.04, 0.7) to (1.04,0.7);
		\draw (-1.25,-0.12) to (-0.22, -0.12);
		\draw (0.22, -0.12) to (1.25,-0.12);
		\draw (-0.97,-0.8) to (0.97,-0.8);
		\draw (-0.74, 1) to (0.74, 1);
		\draw (-1.25, 0.12) to (-0.22, 0.12);
		\draw (0.22,0.12) to (1.25, 0.12);
		\draw [dashed] (-1.04, 0.7) .. controls (0.1,1.2) and (1.8,1.4) .. (1.04, 0.7);
		\draw [dashed] (-1.25,-0.12) .. controls (-0.4,0.4) and (2,0.5) .. (1.25, -0.12);
		\draw [dashed] (-1.25,0.12) .. controls (-0.4,0.7) and (2,0.8) .. (1.25,0.12);
		\draw [dashed] (-1.1,-0.6) .. controls (-0.2, 0) and (1.8,0.1) .. (1.1,-0.6);
		\draw [dashed] (-0.97, -0.8) .. controls (-0.2,-0.3) and (1.7,-0.2) .. (0.97, -0.8);
		\draw (-1.1, -0.6) to (1.1, -0.6);

		\node at (2.7,1.6) {a + d};
		\node at (2.7,0.9) {1 circle};
		\node at (2.7,0.1) {1 circle};
		\node at (2.7,-0.7) {1 circle};

		\draw (6,0) circle (1.25) circle (0.25);
		\draw (6-1.04, 0.7) to (7.04,0.7);
		\draw (6-1.25,-0.12) to (6-0.22, -0.12);
		\draw (6.22, -0.12) to (7.25,-0.12);
		\draw (6-0.97,-0.8) to (6.97,-0.8);
		\draw (6-0.74, 1) to (6.74, 1);
		\draw (6-1.25, 0.12) to (6-0.22, 0.12);
		\draw (6.22,0.12) to (7.25, 0.12);
		\draw (6-1.1, -0.6) to (7.1, -0.6);
		\draw [dashed] (6-1.04, 0.7) .. controls (6.1,1.2) and (7.8,1.4) .. (7.04, 0.7);
		\draw [dashed] (4.9, -0.6) .. controls (5.6,-0.2) and (7.7,0) .. (7.1, -0.6);
		\draw [dashed] (5,-0.8) .. controls (5.8,-0.4) and (7.6,-0.3) .. (7,-0.8);
		\draw [dashed] (4.75,-0.1) .. controls (5.5,0.4) and (5.9,0.3) .. (5.8,-0.1);
		\draw [dashed] (6.25,-0.1) .. controls (7,0.4) and (7.6,0.3) .. (7.25,-0.1);
		\draw [dashed] (4.75, 0.1) .. controls (5.5,0.7) and (6.2,0.5) .. (5.75, 0.1);
		\draw [dashed] (6.24,0.12) .. controls (7, 0.7) and (7.7, 0.5) .. (7.25, 0.12);

		\node at (8.7,1.6) {b};
		\node at (8.7,0.9) {1 circle};
		\node at (8.7,0.1) {2 circles};
		\node at (8.7,-0.7) {1 circle};
		\node at (0,-1.5) {$t=0$};
		\node at (6,-1.5) {$t=0$};

		\draw (0,-4) circle (1.25) circle (0.25);
		\draw (-0.74, -3) to (0.74, -3);
		\draw (-0.25, -4) to (-0.6, -4);
		\draw (0.25, -4) to (0.6, -4);
		\draw (-0.6,-4) arc (180:0:0.6);
		\draw (-1.25, -4) to (-0.9, -4);
		\draw (0.9, -4) to (1.25, -4);
		\draw (-0.7, -3.8) arc (170:10:0.7);
		\draw (-0.9,-4) .. controls (-0.75,-3.95) .. (-0.7, -3.8);
		\draw (0.9,-4) .. controls (0.73,-3.96) .. (0.675, -3.78);
		\draw (-1.23,-4.2) to (-0.82,-4.2);
		\draw (-0.57,-4.2) to (-0.15, -4.2);
		\draw (-0.82,-4.2) to [out=0,in=255] (-0.77,-4.1) to [out=75,in=180] (-0.695,-3.98) to [out=0,in=115] (-0.63,-4.1) to [out=-75,in=180] (-0.57,-4.2);
		\draw (1.23,-4.2) to (0.82,-4.2);
		\draw (0.57,-4.2) to (0.15, -4.2);
		\draw (0.82,-4.2) to [out=180,in=-75] (0.77,-4.1) to [out=115,in=0] (0.695,-3.98) to [out=180,in=75] (0.63,-4.1) to [out=-115,in=0] (0.57,-4.2);
		\draw (-1.23,-3.8) to (-0.95, -3.8);
		\draw (1.23,-3.8) to (0.95, -3.8);
		\draw (-0.75, -3.6) arc (170:10:0.75 and 0.65);
		\draw (-0.95,-3.8) .. controls (-0.8,-3.75) .. (-0.75,-3.6);
		\draw (0.95,-3.8) .. controls (0.8,-3.75) .. (0.73,-3.6);
		\draw (-1.1,-4.6) to (1.1,-4.6);
		\draw (-0.96, -4.8) to (0.96, -4.8);
		\draw (0.15,-3.8) arc (-20:200:0.17 and 0.12);

		\node at (2.7,1.6-4) {a + d};
		\node at (2.7,0.9-4) {1 circle};
		\node at (2.7,0.1-4) {2 circles};
		\node at (2.7,-4.7) {1 circle};

		\draw [decorate,decoration={brace,amplitude=3pt}] (1.8,-3.6) -- (1.8,-4.3);
		\node at (0,-5.5) {$t=1/2$};

		\node at (6,-2.4) {b};
		\node at (6,-3.1) {1 circle};
		\node at (6,-3.9) {2 circles};
		\node at (6,-4.7) {1 circle};
		\draw [decorate,decoration={brace,amplitude=3pt}] (5.1,-2.6) -- (5.1,-3.6);
		\draw [decorate,decoration={brace,amplitude=3pt}] (5.1,-3.63) -- (5.1,-4.17);
		\draw [decorate,decoration={brace,amplitude=3pt}] (5.1,-4.2) -- (5.1,-5.2);

		\node at (6,-6.4) {b};
		\node at (6,-7.1) {1 circle};
		\node at (6,-7.9) {2 circles};
		\node at (6,-8.7) {1 circle};
		\draw [decorate,decoration={brace,amplitude=3pt}] (5.1,-6.6) -- (5.1,-7.6);
		\draw [decorate,decoration={brace,amplitude=3pt}] (5.1,-7.63) -- (5.1,-8.17);
		\draw [decorate,decoration={brace,amplitude=3pt}] (5.1,-8.2) -- (5.1,-9.2);

		\draw (0,-8) circle (1.25) circle (0.25);
		\draw (-0.86, -8.9) to (0.86, -8.9);
		\draw (-0.86, -7.1) to (0.86, -7.1);
		\draw (0.15,-7.8) arc (-20:200:0.17 and 0.12);
		\draw (0.15,-8.2) arc (20:-200:0.17 and 0.12);
		\draw (-1.23,-7.8) to (-0.85,-7.8);
		\draw (1.23,-7.8) to (0.85,-7.8);
		\draw (-0.65,-7.6) arc (170:10:0.65 and 0.55);
		\draw (-0.85,-7.8) .. controls (-0.7,-7.75) .. (-0.65,-7.6);
		\draw (0.85,-7.8) .. controls (0.7,-7.75) .. (0.63,-7.6);

		\draw (-1.23,-8.2) to (-0.85,-8.2);
		\draw (1.23,-8.2) to (0.85,-8.2);
		\draw (-0.65,-8.4) arc (-170:-10:0.65 and 0.55);
		\draw (-0.85,-8.2) .. controls (-0.7,-8.25) .. (-0.65,-8.4);
		\draw (0.85,-8.2) .. controls (0.7,-8.25) .. (0.63,-8.4);

		\draw (-1.23,-8.1) to (-0.85,-8.1);
		\draw (1.23,-8.1) to (0.85,-8.1);
		\draw (-0.6,-8.25) arc (-170:-10:0.6 and 0.6);
		\draw (-0.85,-8.1) .. controls (-0.67,-8.15) .. (-0.6,-8.25);
		\draw (0.85,-8.1) .. controls (0.67,-8.15) .. (0.58,-8.25);
		\draw (1.25, -8) to (0.8,-8) to [out=180,in=45] (0.55, -8.1) to [out=225, in=90] (0.45, -8.3);
		\draw (0.45,-8.3) arc (0:-160:0.45 and 0.33);
		\draw (-0.42, -8.42) to [out=125,in=-90] (-0.58,-7.9);
		\draw (-0.58, -7.9) arc (180:90:0.58 and 0.4);
		\draw (0, -7.5) to [out=0, in=135] (0.3, -7.6) to[out=-45,in=0] (0.25, -8);
		\draw (-1.25, -7.9) to (-0.85,-7.9);
		\draw (-0.22, -7.89) to [out=160,in=90] (-0.45,-8.1) to [out=-90,in=180] (0,-8.52) to [out=0,in=-90] (0.45,-7.9) to [out=90,in=0] (0,-7.3) to[out=180,in=70] (-0.55,-7.6);
		\draw (-0.55,-7.6) .. controls (-0.64, -7.85) .. (-0.85,-7.9);
		\draw (-0.22,-8.1) arc (155:380:0.24 and 0.24);

		\node at (0,-9.5) {$t=1$};

		\node at (2.7,1.6-8) {a + d};
		\node at (2.7,-8) {1 circle};
		\draw [decorate,decoration={brace,amplitude=7pt}] (1.8,-6.7) -- (1.8,-9.3);
	\end{tikzpicture}
	\caption{}\label{displace}
\end{figure}

Inspecting Figure \ref{displace}, the interpolation $\lbar{h}$ has all generic preimages completing globally to either one or two circles.

The exceptional cases where two circles never occur are circle bundles over finite planar surfaces $F$ which degenerate to one point fibers over $\de F$. These are stated 3-manifolds.
\qed

\section{Thoughts on the Goeritz Group}

Let us define the genus $G$ Goeritz group $G_g$ to be $G_g = \pi_1(\text{Heegaard}_g)$ where $\text{Heegaard}_g$ is the space of genus $g$ Heegaard decompositions of $S^3$. $G_0 \cong \Z_2$, $G_1 \cong \Z_4$, and an explicit finite presentation of $G_2$ is given in \cite{}. In \cite{dehn} it is confirmed that five generators\footnote{Martin Scharlemann noticed in 2018 (private communication) that one of the proposed parallel generators is actually redundant being a word in three of the remaining four.} proposed to suffice for each $G_g$ in fact do generate $G_3$. For $g \geq 4$ it is not known if $G_g$ is finitely generated. This paper was motivated by a strategy for proving $G_g$ to be infinitely generated for $g \geq 4$. For every element $h \in G_g$, we describe below a closed 3-manfiold $M_h$ together with an embedding $e_h: M_h \hookrightarrow \R^4$.

\begin{thm}\label{old3}
	If the function $\mathrm{2}\text{-}\mathrm{width}(e_h)$, $h \in G_g$, diverges, then $G_g$ is infinitely generated.
\end{thm}

There is a related equivalence which we will return to.

\begin{thm}
	$G_g$ is infinitely generated iff the "frame complexity" $\mathrm{fc}(h)$ diverges over $h \in G_g$.
\end{thm}

\begin{proof}[Proof of Theorem \ref{old3}]
	Given a loop $h$ of Heegaard surfaces $\Sigma_\te \subset \R^3$, $0 \leq \te \leq 1$, we may consider the normal bundle $S^1 \times \R^3$ to a circle embedded in $\R^4$ by $\te \mapsto (\sin(2\pi \te), \cos(2\pi\te), 0, 0)$ for $0 \leq \te \leq 1$. Trivializing the normal bundle $\{\Sigma_t, 0 \leq t \leq 1\}$ now defines a circular family of Heegaard surfaces in $\R^4$ which together form a 3-manifold $M_h$, the $h$-mapping torus of $\Sigma$, with an embedding $e_h: M_h \hookrightarrow \R^4$. Consider the projection $\pi: \R^4 \ra \R^2$ preserving the first and last coordinate, $\pi(w,x,y,z) = (w,z)$. If it is true that $G_g$ is finitely generated (fg) then for each generator $g_i$ we may realize it as a specific isotopy $\lbar{g}_i$ of $\Sigma_g$ beginning at ending at the same "base point" Heegaard embedding $\Sigma_g \subset \R^3$, shown in Figure \ref{twists}. The height coordinate $z$ on $\R^3$ is assumed to agree with the 4th coordinate of $\R^4$ under our embedding $S^1 \times \R^3 \hookrightarrow \R^4$. The isotopy $\lbar{g}_i$ will have at any given time a maximum number of circles in a generic level; let $n_i$ be the maximum of this number over all times during the isotopy. Finally let $n = \max_i n_i$. It is clear now that any word in $\{g_i, g^{-1}_i\}$ can be written as an isotopy in which the maximum number of circles in a generic level is $\leq n$.

	If the general element $h \in G_g$, $h = \mathrm{word}(g_i, g_i^{-1})$ of length $l$ is written out in terms of $\lbar{g}_i$ and $\lbar{g}_i^{-1}$ normal to the circle $(\sin(\te), \cos(\te), 0, 0)$, $0 \leq \te \leq \pi$, with the $j$th letter occupying $j\frac{\pi}{l} \leq \te \leq (j+1)\frac{\pi}{l}$, and with $\id$ for $\pi \leq \te \leq 2\pi$ (cyclic boundary conditions), then the maximum number of circles seen in a regular preimate $\pi^{-1}(p)$ will be $n+2$ ($n$ circles for $0 \leq \te \leq \pi$ and 2 circles for $\pi \leq \te \leq 2\pi$). Thus $n+2$ serves as a global upper bound to 2-width$(e_h)$. If 2-width instead diverges on $e_h$, $h \in G_g$, it must be that $G_g$ was not finitely generated, as assumed. This proves Theorem \ref{old3}.
\end{proof}

Since it is not yet clear if 2-width is the right tool to establish infinite generation of $G_g$ (if true), we offer Theorem \ref{old4} as a fully general criterion. It is easiest here to work with piecewise linear Heegaard surfaces. Given a smooth Heegaard $\Sigma^2 \subset \R^3$ after a homothetic rescaling, one can always select a Whitehead triangulation so that fixing the vertices and pulling the edges taut to straight line segments, results in a piecewise smooth isotopy and the final piecewise linear surface consists of triangles whose geometry comes from a compact parameter space $X$ and so that the dihedral angles are all in the range $[\pi - \epsilon, \pi + \epsilon]$ for a fixed small $\epsilon \geq 0$. For example, $X$ could be the set of triangles with edge length in the range $[1, 1.5]$. If we add to these the restriction that the open neighborhood of radius $\epsilon$ is PL normal bundle and fix $\epsilon$ above once and for all, we thus have a class of PL embeddings, call them $\mathrm{Good}_n$, consisting of $n$ or fewer vertices, so that

\begin{lemma}
	For a fixed $n$, the space of good Heegaard embeddings in $\R^3$, $\mathrm{Good}_n$ is compact, module translations of $\R^3$.
\end{lemma}

\begin{proof}
	Straightforward.
\end{proof}

An immediate consequence of compactness is:

\begin{lemma}\label{old44}
	For every sufficiently large positive integer $n$ there is a larger positive integer $N$ so that every element of $\mathrm{Good}_n$ is connected to a standard "base point" triangulation $\ast$ within $\mathrm{Good}_N$ by a path in $\mathrm{Good}_N$. (A "path" allows re-triangulation but must be continuous in the Hausdorff topology.) \qed
\end{lemma}

\begin{defin*}
	The \emph{frame complexity} (fc) of an element $\alpha \in G_g$ is the minimum $n$ so that $\alpha$ is represented by a loop of PL embeddings $\{\alpha_t\}$ so that each $\alpha_t$ is in $\mathrm{Good}_n$.
\end{defin*}

\begin{thm}\label{old4}
	$G_g$ is finitely generated (fg) iff frame complexity (fc) is uniformly bounded above on $G_g$.
\end{thm}

\begin{proof}[Proof of Theorem \ref{old4}]
	$G_g$ fg clearly implies bounded fc. For every $\alpha \in G_g$ it is simply the max of fc over a generating set.

	On the other hand, if we can homotop the loop $\alpha$ so that, without changing the notation $\alpha$, every "frame" $\alpha_t$ of our movie lies in $\mathrm{Good}_n$, we can surely write it as a product from a fixed generating set. This is done using lemma \ref{old44}. A \emph{long} path in $\mathrm{Good}_n$ can be intermittently modified by inserting an arc $\gamma_i$ to $\ast$ and then, immediately, its inverse $\gamma_i^{-1}$, $\gamma_i \in \mathrm{Good}_N$. By compactness of $\mathrm{Good}_N$ there will be, up to deformation, only finitely many choices for the units formed by ($\gamma_i \circ \text{segment}_i \circ \gamma_i^{-1}$), where $\text{segment}_i$ represents the intervals of the original long path between the intermittent returns to $\ast$. These "units" form the required finite generating set. The $\text{segments}_i$ and the $\gamma_i$ may (and must) all be chosen from a compact subset of the space of paths in $\mathrm{Good}_n$ and $\mathrm{Good}_N$, respectively.
\end{proof}

\end{document}